\documentclass[12pt]{amsart}

\usepackage[english]{babel}
\usepackage[utf8x]{inputenc}
\usepackage[T1]{fontenc}

\usepackage{amsmath}
\usepackage{graphicx}
\usepackage[colorinlistoftodos]{todonotes}
\usepackage[colorlinks=true, allcolors=blue]{hyperref}
\usepackage{float} %Allows us to stick tables where we want them.
\usepackage[inner=1.0in,outer=1.0in,bottom=1.0in, top=1.0in]{geometry}

\usepackage{amssymb,amsfonts, amsthm,enumerate} % standard math packages

\usepackage{verbatim}
\usepackage{graphicx, graphics}
\usepackage{longtable} 
\usepackage{amscd}

\usepackage[all,arc,curve,frame,color]{xy}
\usepackage{subfigure}
\usepackage{url}

\usepackage{tikz}
\usepackage{tikz-cd}
\usetikzlibrary{arrows}
\tikzstyle{block}=[draw opacity=0.7,line width=1.4cm]

\newtheorem{thm}{Theorem}[section]
\newtheorem{lem}[thm]{Lemma}
\newtheorem{cor}[thm]{Corollary}

\newtheorem*{conjecture*}{Conjecture}
\newtheorem*{thm*}{Theorem}

\newtheorem{question}[thm]{Question}

\theoremstyle{definition}

\newtheorem{remark}[thm]{Remark}

\newtheorem{example}[thm]{Example}

\usepackage{xcolor}
\newcommand{\OO}{\mathcal{O}}    % Fancy script O
\newcommand{\FF}{\mathbb{F}}      % Finite field
\newcommand{\ZZ}{\mathbb{Z}}     % Integers
\newcommand{\PP}{\mathbb{P}}      % Projective space
\newcommand{\QQ}{\mathbb{Q}}      %Rationals
\newcommand{\an}[1]{\operatorname{an}}  % analytic space notation
\newcommand{\ve}{\varepsilon}
\newcommand{\p}{\mathfrak{p}}

\DeclareMathOperator{\Per}{Per}

\DeclareMathOperator{\lcm}{lcm}

\title{Density of Periodic Points\\ for Latt\`es maps over Finite Fields}
\author{Zo\"e Bell}
\address[Z.~Bell]{Harvey Mudd College\\ Claremont, CA 91711}
\email{zbell@g.hmc.edu}

\author{Jasmine Camero}
\address[J.~Camero]{Dept. of Mathematics  \\
Emory University \\
Atlanta, GA 30322}
\email{jasmine.camero@emory.edu}

\author{Karina Cho}
\address[K.~Cho]{Dept. of Mathematics\\ Stony Brook University\\ Stony Brook, NY, 11794-3651}
\email{karina.cho@stonybrook.edu}

\author{Trevor Hyde}
\address[T.~Hyde]{Dept. of Mathematics\\
University of Chicago \\
Chicago, IL 60637}
\email{tghyde@uchicago.edu}

\author{Chieh-Mi Lu}
\address[F.~Lu]{California Sate University San Marcos\\
333 S Twin Oaks Valley Rd, 
San Marcos, CA 92096}
\email{clu@csusm.edu }

\author{Rebecca Miller}
\address[R.~Miller]{Dept. of Mathematics and Statistics\\ Saint Louis University, MO, 63103}
\email{r.lauren.miller@slu.edu}

\author{Bianca Thompson}
\address[B.~Thompson]{Dept. of Mathematics\\Westminster College\\
1840 S 1300 E\\
 Salt Lake City, UT 84105, USA} 
 \email{bthompson@westminstercollege.edu}

\author{Eric Zhu}
\address[E.~Zhu]{Dept. of Mathematics \\ Georgia Institute of Technology \\ Atlanta, GA 30332}
\email{ezhu31@gatech.edu}

\begin{document}

\begin{abstract}
    Let $L_d$ be the Latt\`es map associated to the multiplication-by-$d$ endomorphism of an elliptic curve $E$ defined over a finite field $\FF_q$. We determine the density $\delta(L_d,q)$ of periodic points for $L_d$ in $\PP^1(\FF_q)$. We show that the periodic point densities $\delta(L_d,q^n)$ converge as $n \rightarrow \infty$ along certain arithmetic progressions, and compute simple explicit formulas for $\delta(L_\ell,q)$ when $\ell$ is a prime and $E$ belongs to a special family of supersingular elliptic curves.
\end{abstract}

\maketitle

\section{Introduction} 
\label{sec introduction}

Let $\FF_q$ be a finite field with $q$ elements and suppose that $f(x) \in \FF_q(x)$ is a rational function. The projective line $\PP^1(\FF_{q})$ is a finite set closed under iteration of $a \mapsto f(a)$. We say that a point $a \in \PP^1(\FF_{q})$ is \emph{periodic} under $f$ if $f^m(a) = a$ for some $m$, where $f^m$ denotes the $m$-fold composition of $f$ with itself.

\begin{question}
\label{question intro}
What is the probability $\delta(f,q)$ that a randomly chosen element of $\PP^1(\FF_{q})$ is periodic under iteration of $f$? 
\end{question}

In this paper we answer Question \ref{question intro} for semiconjugates of elliptic curve endomorphisms, also known as Latt\`es maps. See Section \ref{sec ell curves} for background on elliptic curves and Latt\`es maps. Our first result determines the density of periodic points for Latt\`es maps over an arbitrary finite field $\FF_q$.  Some notation: Let $\Per(f,\FF_q)$ denote the subset of points in $\PP^1(\FF_q)$ that are periodic under the rational function $f(x)$. If $\ell$ is a prime and $k$ is an integer, then $v_\ell(k)$ denotes the multiplicity of $\ell$ as a factor of $k$.

\begin{thm}
\label{thm intro lattes periodic count}
Let $E$ be an elliptic curve defined over $\FF_q$ and let $\tau$ be the integer defined by $\#E(\FF_q) = q + 1 - \tau$. If $d$ is an integer coprime to $q$ and $L_d: \PP^1 \rightarrow \PP^1$ is the Latt\`es map associated to the multiplication-by-$d$ map on $E$, then

\[
    \delta(L_d,q) := \frac{\#\mathrm{Per}(L_d,\FF_{q})}{\#\PP^1(\FF_q)} = \frac{1}{2}\Big(\frac{1}{\pi_+} + \frac{1}{\pi_-}\Big) + \frac{\tau}{2(q + 1)}\Big(\frac{1}{\pi_+} - \frac{1}{\pi_-}\Big),
\]
where
\[
    \pi_\pm := \prod_{\ell \mid d} \ell^{v_\ell(q + 1 \pm \tau)},
\]
and the product is taken over all primes $\ell$ dividing $d$. Furthermore,
\[
    \Big|\delta(L_d,q) - \frac{1}{2}\Big(\frac{1}{\pi_+} + \frac{1}{\pi_-}\Big)\Big| < \frac{1}{q^{1/2} + q^{-1/2}}.
\]
\end{thm}

Using Theorem \ref{thm intro lattes periodic count} we study the behavior of the periodic point density $\delta(L_d, q^n)$ as $n$ varies. Our second result shows that $\delta(L_d,q^n)$ converges as $n \rightarrow \infty$ along certain arithmetic progressions.

\begin{thm}
\label{thm intro periodic proportions}
Let $E$ be an elliptic curve defined over $\FF_q$, let $d$ be an integer coprime to $q$, and let $\tau_n$ be the sequence of integers such that $\#E(\FF_{q^n}) = q^n + 1 - \tau_n$. Then for each $n\geq 1$ there exists some $N$ depending on $q, \tau_1, d, n$ such that
\[
    \lim_{\substack{m\rightarrow \infty\\ m \equiv n \bmod c d^N}} \delta(L_d,q^m) = \frac{1}{2}\Big(\frac{1}{\pi_{n,+}} + \frac{1}{\pi_{n,-}}\Big),
\]
where
\[
    \pi_{n, \pm} := \prod_{\ell \mid d} \ell^{v_\ell(q^n + 1 \pm \tau_n)},
\]
the product is taken over all primes $\ell$ dividing $d$, $c := \lcm(\ell^2 - 1 : \ell \mid d \text{ is prime})$, and the limit is taken over all positive integers $m$ such that $m \equiv n \bmod c d^N$.
\end{thm}

Then $N$ in Theorem \ref{thm intro periodic proportions} is, in principle, effectively computable as it comes from the modulus of continuity of several explicit exponential functions $n \mapsto \gamma^m$ with respect to the $\ell$-adic metric for primes $\ell$ dividing $d$. Theorem \ref{thm intro periodic proportions} implies that any density that occurs for at least one $n$ must occur infinitely often. 

Our final result applies Theorem \ref{thm intro lattes periodic count} to $E$ in a special family of supersingular elliptic curves to get even more explicit formulas for the density $\delta(L_\ell,q^n)$ with $\ell$ prime.

\begin{thm}
\label{thm intro special supersingular evaluation}
Let $E$ be an elliptic curve over $\FF_q$ such that $\#E(\FF_q) = q + 1$. If $\ell > 2$ is a prime not dividing $q$ and $e$ is the multiplicative order of $q$ modulo $\ell$, set 
\[
    w_n := v_\ell(q^e - 1) + v_\ell(n).
\]
Then
\[
    \delta(L_\ell,q^n)
    =
    \begin{cases}
         \ell^{-w_n} & \text{if $n$ is odd, $e \mid 2n$, and $e \nmid n$,}\\
         \frac{1}{2}(1 + \ell^{-2w_n}) + \frac{\epsilon}{q^{n/2} + q^{-n/2}}(1 - \ell^{-2w_n}) & \text{if $n$ is even and $e \mid n$,}\\
         1 & \text{otherwise,}
    \end{cases}
\]
where $\epsilon$ is a sign determined by
\[
\epsilon = 
    \begin{cases}
        +1 & \text{if $n\equiv 2 \bmod 4$ and $e\nmid n/2$, or $n\equiv 0 \bmod 4$ and $e \mid n/2$,} \\
        -1 & \text{if $n\equiv 2 \bmod 4$ and $e\mid n/2$, or $n\equiv 0 \bmod 4$ and $e \nmid n/2$.}
    \end{cases}
\]
Hence if $e$ does not divide $2n$, then $L_\ell$ induces a permutation on $\PP^1(\FF_{q^n})$.
\end{thm}

\subsection{Related work}

The most tractable cases of Question \ref{question intro} are when the rational function $f$ is a semiconjugate of an endomorphism of an algebraic group. In this case, one can translate questions about the dynamics of $f$ into questions about the underlying group structure, which are easier to analyze. Such semiconjugates include power maps, Chebyshev polynomials, Dickson polynomials, and Latt\`es maps. Periodic densities for power maps (and monomial maps more generally) were studied in Hu, Sha \cite{shahu}, and for power maps and Chebyshev polynomials by Manes, Thompson \cite{manes2013periodic}. In this paper we treat the Latt\`es case.

Juul, Kurlberg, Madhu, Tucker \cite{JKMT} studied the density of periodic points for the modulo $\p$ reduction of a fixed rational function $f(x)$ defined over a global field $K$ as $\p$ varies through the primes in $\OO_K$. They show that the limsup of the density of periodic points of $f$ can be made arbitrarily small as the norm of $\p$ tends to infinity by choosing $f$ from a Zariski dense open set of degree $d$ rational functions \cite[Thm. 1.2]{JKMT}. In their Example 7.3 they fix an elliptic curve $E$ defined over $\QQ$ and a Latt\`es map $L_\ell$ with $\ell$ prime and show that in many cases the liminf of the periodic density of $L_\ell$ modulo $\p$ converges to 0 as the norm of $\p$ tends to infinity. Note that these limits are in a different direction from those we take in Theorem \ref{thm intro periodic proportions}. See also the earlier work of Madhu \cite{madhu} and the subsequent work of Juul \cite{juul}.

Ugolini \cite{ugolini} studied the phase portraits of Latt\`es maps associated to ordinary elliptic curves over finite fields. The methods of \cite{ugolini} are substantively equivalent to those we use to prove Theorem \ref{thm intro lattes periodic count}, but there appears to be no direct overlap in results.

K\"u\c{c}\"uksakalli \cite{kucuksakalli} analyzed the value sets of Latt\`es maps over finite fields, arriving at a characterization of when a given Latt\`es map $L_d$ induces a permutation of $\PP^1(\FF_q)$, which is equivalent to our Corollary \ref{cor permutation characterization} (see Remark \ref{remark kucuksakalli}.)

\subsection{Acknowledgements}
This project began as part of the Summer@ICERM 2019 REU. We thank ICERM for hosting us and providing us access to their computational resources. We thank Joe Silverman for generously sharing insights that led us to the proof of Theorem \ref{thm intro lattes periodic count}. We also wish to thank Women in Sage: Sagedays 90, where this project was originally developed. Trevor Hyde is partially supported by the NSF MSPRF and the Jump Trading Mathlab Research Fund.

\section{Density of periodic points for Latt\`es maps}
\label{sec ell curves}
The goal of this section is to prove the following theorem.

\begin{thm}
\label{thm lattes periodic count}
Let $E$ be an elliptic curve defined over $\FF_q$ with trace of Frobenius $\tau$. If $d$ is an integer coprime to $q$ and $L_d: \PP^1 \rightarrow \PP^1$ is the Latt\`es map associated to the multiplication-by-$d$ map on $E$, then

\[
    \delta(L_d,q) := \frac{\#\mathrm{Per}(L_d,\FF_{q})}{\#\PP^1(\FF_q)} = \frac{1}{2}\Big(\frac{1}{\pi_+} + \frac{1}{\pi_-}\Big) + \frac{\tau}{2(q + 1)}\Big(\frac{1}{\pi_+} - \frac{1}{\pi_-}\Big),
\]
where
\[
    \pi_\pm := \prod_{\ell \mid d} \ell^{v_\ell(q + 1 \pm \tau)},
\]
and the product is taken over all primes $\ell$ dividing $d$. Furthermore,
\[
    \Big|\delta(L_d,q) - \frac{1}{2}\Big(\frac{1}{\pi_+} + \frac{1}{\pi_-}\Big)\Big| < \frac{1}{q^{1/2} + q^{-1/2}}.
\]
\end{thm}

We begin by reviewing some of the basics of elliptic curves over finite fields and their Latt\`es maps. For a more complete treatment we refer the reader to Silverman \cite{jhsec}, and Chapter V in particular.

\subsection{Elliptic curves}
An \emph{elliptic curve} defined over a field $K$, denoted $E/K$, is a smooth, projective, algebraic curve $E$ of genus 1 defined over $K$ together with a prescribed \emph{base point} $O \in E(K)$. There is a natural commutative algebraic group law on $E$ defined over $K$ with $O$ as the identity element. In particular, if $L/K$ is any field extension, the $L$-points of $E$ form a group. If $P, Q \in E(\overline{K})$ are points on $E$, then we write $P + Q$ and $-P$ for the sum and inverse, respectively.

We say two elliptic curves $E_1$, $E_2$ defined over $K$ are isomorphic if there is an invertible map $f: E_1 \rightarrow E_2$ of algebraic curves that respects base points $f(O_1) = O_2$. If $f$ is an isomorphism, then $f: E_1(M) \xrightarrow{\sim} E_2(M)$ gives a group isomorphism between the $M$-points on the curves for any extension $M/K$ \cite[Thm. III.4.8]{jhsec}.

Every elliptic curve $E$ over a field $K$ has a model as a plane curve given by a \emph{Weierstrass equation} of the form
\begin{equation}
\label{eqn weierstrass}
    E: y^2 + A_1xy + A_3y = x^3 + A_2x^2 + A_4x + A_6,
\end{equation}
where $A_i \in K$ \cite[\textsection III.1]{jhsec}. If the characteristic of $K$ is at least 5, then there is a reduced Weierstrass equation of the form
\[
    E: y^2 = x^3 + Ax + B.
\]
However, since we are primarily interested in elliptic curves over arbitrary finite fields, we use the Weierstrass equations of the form \eqref{eqn weierstrass} for full generality.

Weierstrass curves intersect the line at infinity in $\PP^2(K)$ at exactly one point expressed in homogeneous coordinates as $(0 : 1 : 0)$. By convention, we let the base point $O$ of $E$ be this unique point at infinity.

If $P = (a,b)$ is a point on $E$ expressed in affine coordinates, then the additive inverse of $P$ is
\begin{equation}
\label{eqn negative P}
    {-P} := (a, -b - A_1a - A_3).
\end{equation}
Let $x: E \rightarrow \PP^1$ be the $x$-coordinate projection $x(a,b) = a$. Since there are clearly at most 2 points on $E$ with the same $x$-coordinate, it follows that $x(P) = x(Q)$ if and only if $P = \pm Q$ \cite[\textsection III.2]{jhsec}.

\subsection{Elliptic curves over $\FF_q$}
Let $\FF_q$ be a finite field with $q$ elements, where $q$ is a prime power. The \emph{Frobenius automorphism} $F: \PP^2(\overline{\FF}_q) \rightarrow \PP^2(\overline{\FF}_q)$ is the map defined on projective coordinates by
\[
    F(a : b : c) = (a^q : b^q : c^q).
\]
If $E \subseteq \PP^2$ is an elliptic curve defined over $\FF_q$, then $F$ restricts to an automorphism of $E$. Furthermore, Galois theory implies that $P \in E(\overline{\FF}_q)$ belongs to $E(\FF_{q^n})$ if and only if $F^n(P) = P$.

If $E$ is an elliptic curve over $\FF_q$, then $E(\FF_q)$ is a finite abelian group. Let $\tau$ be the unique integer such that
\[
    \#E(\FF_q) = q + 1 - \tau.
\]
The integer $\tau$ is called the \emph{trace of Frobenius} for $E$ over $\FF_q$. This name comes from the fact that $\tau$ may be realized as the trace of the action of $F$ as a linear endomorphism of the $\ell$-adic Tate module associated to $E$ \cite[Rmk. V.2.6]{jhsec}.

\subsection{Latt\`es maps}
For each $d \in \ZZ$, the \emph{multiplication-by-$d$} map $[d]: P \mapsto dP$ is a group endomorphism of $E(\FF_q)$. Since $E(\FF_q)$ is abelian, $[d]$ commutes with the inverse map $[-1]$. Galois theory implies that there exists a rational function $L_d(x) \in \FF_q(x)$ such that the following diagram commutes,
\[
    \begin{tikzcd}
        E \arrow[d,"x",swap] \arrow[r,"{[d]}"] & E\arrow[d,"x"]\\
        \PP^1 \arrow[r,"L_d"] & \PP^1.
    \end{tikzcd}
\]
That is, for all points $P \in E(\overline{\FF}_q)$,
\[
    L_d(x(P)) = x(dP).
\]
Since every $a \in \PP^1(\overline{\FF}_q)$ is the $x$-coordinate of some point $P \in E(\overline{\FF}_q)$, this identity completely determines $L_d(x)$. If $d, e \in \ZZ$, then $[de] = [d] \circ [e]$, hence the above diagram implies that $L_{de} = L_d \circ L_e$.

The rational function $L_d$ is called the \emph{$d$th Latt\`es map} associated to $E$. We are interested in the dynamics of Latt\`es maps. The conjugacy class of $L_d$ is determined by the isomorphism class of $E$. Hence the dynamical properties of $L_d$ are independent of our choice of model for $E$ over $\FF_q$. For simplicity, when we talk about Latt\`es maps for an elliptic curve $E$, we will always mean with respect to a Weierstrass model for $E$ as in \eqref{eqn weierstrass}. For more background on Latt\`es maps and their dynamics, see Silverman \cite[Chp. 6]{ads}.

The following lemma characterizes the periodic points of $L_d$ in terms of the group structure on $E$.

\begin{lem}
\label{lem characterizing lattes periodic points}
Let $E/\FF_q$ be an elliptic curve and let $L_d: \PP^1 \rightarrow \PP^1$ be the $d$th Latt\`es map associated to $E$. Then $a \in \PP^1(\FF_q)$ is periodic under $L_d$ if and only if there is some point $P \in E(\FF_{q^2})$ such that $a = x(P)$ and $P$ has order coprime to $d$.
\end{lem}

\begin{proof}
Since $\infty \in \PP^1(\FF_q)$ is the $x$-coordinate of the identity $O$ on $E(\FF_q)$ (which has order 1), it follows from the definition of Latt\`es maps that
\[
    L(\infty) = L(x(O)) = x(dO) = x(O) = \infty.
\]
Hence $\infty$ is always periodic under $L$.

Suppose that $E$ has a Weierstrass equation as in \eqref{eqn weierstrass}. Then for each $a \in \FF_q$, the quadratic equation
\[
    y^2 + A_1ay + A_3y = a^3 + A_2a^2 + A_4a + A_6,
\]
splits completely in $\FF_{q^2}$. Hence every $a \in \FF_q$ is the $x$-coordinate of some point on $E(\FF_{q^2})$. Note that since $\FF_{q^2}$ is a finite field, $E(\FF_{q^2})$ is a finite abelian group. Suppose that $a = x(P)$ where $P \in E(\FF_{q^2})$ has order $n$. Then $a$ is periodic under $L$ if and only if there is some $k$ such that $L^k(a) = a$, which is equivalent to
\[
    a = L^k(a) = L^k(x(P)) = x(d^kP).
\]
Thus $d^kP = \pm P$, hence $d^{2k}P = P$, which is equivalent to $d^{2k} \equiv 1 \bmod n$. Such an integer $k$ exists if and only if $d$ is a unit modulo $n$, which is to say that $P$ is a point with order coprime to $d$.
\end{proof}

\subsection{Quadratic twists}
Let $E$ be an elliptic curve over $\FF_q$. The \emph{quadratic twist} of $E$ is an elliptic curve $E'$ over $\FF_q$ characterized up to isomorphism defined over $\FF_q$ by the following property: There exists an isomorphism $\iota: E \rightarrow E'$ defined over $\FF_{q^2}$, and for all such isomorphisms $\iota$ and points $P \in E(\FF_{q^2})$ we have
\begin{equation}
\label{eqn quadratic twist characterization}
    F(\iota(P)) = -\iota(F(P)).
\end{equation}

If the characteristic of $\FF_q$ is at least 5 and
\[
    E: y^2 = x^3 + Ax + B
\]
is a reduced Weierstrass equation for $E$, then
\[
    E': \alpha^2 y^2 = x^3 + Ax + B
\]
is an explicit formula for the quadratic twist $E'$, where $\alpha$ is a primitive element of $\FF_{q^2}$ such that $\alpha^2 \in \FF_q$. An isomorphism $\iota : E \rightarrow E'$ is given in affine coordinates by
\[
    \iota(a,b) = (a,b/\alpha).
\]
Note that by construction, $\alpha^q = -\alpha$. Hence if $P = (a,b) \in E(\FF_{q^2})$, then 
\[
    F(\iota(P)) = F(a,b/\alpha) = (a^q, b^q/\alpha^q) = (a^q, -b^q/\alpha) = -\iota(F(P)).
\]
There are similar explicit formulas for $E'$ in characteristics 2 and 3, but for our purposes the simplest way to work with quadratic twists in arbitrary characteristic is via the characterization \eqref{eqn quadratic twist characterization}. For the general theory of twists see Silverman \cite[\textsection X.2]{jhsec}, and for quadratic twists of elliptic curves in particular, see \cite[Ex. X.2.4]{jhsec}.

\begin{lem}
\label{lem frobenius eigenspaces}
Let $E$ be an elliptic curve defined over $\FF_q$ with quadratic twist $E'$.
Let $A, A' \subseteq E(\FF_{q^2})$ be the subgroups
\begin{align*}
    A &:= \{P \in E(\FF_{q^2}) : F(P) = P\} = E(\FF_q)\\
    A' &:= \{P \in E(\FF_{q^2}) : F(P) = -P\},
\end{align*}
where $F$ is the Frobenius automorphism.
\begin{enumerate}
    \item If $\iota: E \rightarrow E'$ is an isomorphism defined over $\FF_{q^2}$, then the restriction of $\iota$ to $A'$ gives a group isomorphism $\iota: A' \rightarrow E'(\FF_q)$.
    \item $\#E(\FF_q) + \#E'(\FF_q) = 2(q + 1)$.
    \item If $\tau$ is the trace of Frobenius of $E/\FF_q$, then the trace of Frobenius for $E'/\FF_q$ is $-\tau$.
\end{enumerate}
\end{lem}

\begin{proof}
(1) It suffices to show that $\iota(A') \subseteq E'(\FF_q)$ and $\iota^{-1}(E'(\FF_q)) \subseteq A'$.cIf $P \in A'$, then \eqref{eqn quadratic twist characterization} implies
\[
    F(\iota(P)) = -\iota(F(P)) = -\iota(-P) = \iota(P),
\]
hence $\iota(P) \in E'(\FF_q)$ by Galois theory. If $Q \in E'(\FF_{q})$, then 
\[
    F(\iota^{-1}(Q)) = -\iota^{-1}(F(Q)) = -\iota^{-1}(Q),
\]
where the first equality is equivalent to
\eqref{eqn quadratic twist characterization} and the second equality follows from $F(Q) = Q$. Hence $\iota^{-1}(Q) \in A'$.

(2) Consider the map $\widetilde{x}: A \sqcup A' \rightarrow \PP^1(\FF_q)$ from the disjoint union of $A$ and $A'$ to $\PP^1(\FF_q)$ induced by the $x$-coordinate projection map. We claim that $\widetilde{x}$ is exactly 2-to-1. Since $A = E(\FF_q)$ and $A' \cong E'(\FF_q)$ by (1), this will imply that
\[
    \#E(\FF_q) + \#E'(\FF_q) = 2\#\PP^1(\FF_q) = 2(q + 1).
\]
If $a \in \PP^1(\FF_q)$, then as we argued in the proof of Lemma \ref{lem characterizing lattes periodic points}, there exists some $P \in E(\FF_{q^2})$ such that $x(P) = a$. Since $F(P)$ is another point in $E(\FF_{q^2})$ with $x(F(P)) = a^q = a$, it follows that $F(P) = \pm P$. Thus $P \in A$ or $A'$, hence $\widetilde{x}$ is surjective. If $P \neq -P$, then $\pm P$ both belong to exactly one of $A$ or $A'$, and these are precisely the two points in $A \sqcup A'$ mapping to $a$ under $\widetilde{x}$. If $P = -P$, then $P$ is the only point on $E(\FF_{q^2})$ with $x$-coordinate $a$. Since $P \in A \cap A'$, there are two copies of $P$ in $A \sqcup A'$ and these are precisely the two points mapping to $a$ under $\widetilde{x}$. Hence $\widetilde{x}$ is exactly 2-to-1.

(3) Since $\#E(\FF_q) = q + 1 - \tau$, it follows from (2) that
\[
    \#E'(\FF_q) = 2(q + 1) - (q + 1 - \tau) = q + 1 + \tau.
\]
Therefore the trace of Frobenius for $E'/\FF_q$ is $-\tau$.
\end{proof}

We now turn to the proof of Theorem \ref{thm lattes periodic count}.

\begin{proof}[Proof of Theorem \ref{thm lattes periodic count}]
Lemma \ref{lem characterizing lattes periodic points} implies that $a \in \PP^1(\FF_q)$ is periodic under $L_d$ if and only if there exists some point $P \in E(\FF_{q^2})$ with order coprime to $d$ such that $a = x(P)$. Let $A, A' \subseteq E(\FF_{q^2})$ be the subgroups defined in Lemma \ref{lem frobenius eigenspaces}. Then any point $P \in E(\FF_{q^2})$ with $x(P) \in \PP^1(\FF_q)$ must belong to $A \cup A'$.

Let $A_d$ and $A_d'$ denote the subsets of $A$ and $A'$, respectively, of elements with order coprime to $d$. Lemma \ref{lem frobenius eigenspaces} implies that $\#A = \#E(\FF_q) = q + 1 - \tau$ and $\#A' = \#E'(\FF_q) = q + 1 + \tau$. If $G$ is a finite abelian group, then the number of elements of $G$ with order coprime to $d$ is the largest factor of $\#G$ coprime to $d$. Hence
\[
    \#A'_d = \frac{q + 1 + \tau}{\pi_+}, \hspace{.35in} \#A_d = \frac{q + 1 - \tau}{\pi_-},
\]
where
\[
    \pi_{\pm} := \prod_{\ell\mid d}\ell^{v_\ell(q + 1 \pm \tau)}.
\]
Suppose that $A_d' \cap A_d$ consists of $r$ points. The involution $P \mapsto -P$ acts freely on $A_d' \cup A_d \setminus (A_d' \cap A_d)$ and fixes each point of $A_d' \cap A_d$. Thus,
\begin{align*}
    \#\Per(L_d,\FF_q) &= \#x(A_d' \cup A_d)\\
    &= \#x(A_d' \setminus (A_d' \cap A_d)) + \#x(A_d \setminus (A_d' \cap A_d)) + \#x(A_d' \cap A_d')\\
    &= \frac{(q + 1 + \tau)/\pi_+ - r}{2} + \frac{(q + 1 - \tau)/\pi_- - r}{2} + r\\
    &= \frac{q + 1}{2}\Big(\frac{1}{\pi_+} + \frac{1}{\pi_-}\Big) + \frac{\tau}{2}\Big(\frac{1}{\pi_+} - \frac{1}{\pi_-}\Big).
\end{align*}
Dividing by $\#\PP^1(\FF_q) = q + 1$ we conclude that
\[
    \delta(L_d,q) := \frac{\#\Per(L,\FF_q)}{\#\PP^1(\FF_q)} = \frac{1}{2}\Big(\frac{1}{\pi_+} + \frac{1}{\pi_-}\Big) + \frac{\tau}{2(q + 1)}\Big(\frac{1}{\pi_+} - \frac{1}{\pi_-}\Big).
\]
Hasse proved the following bound on the trace of Frobenius,
\begin{equation}
\label{eqn Hasse bound}
    |\tau| \leq 2\sqrt{q}.
\end{equation}
See, for example, \cite[Thm. V.1.1]{jhsec}. Since $|\pi_+^{-1} - \pi_-^{-1}| < 1$ it follows that
\[
    \Big|\delta(L_d,q) - \frac{1}{2}\Big(\frac{1}{\pi_+} + \frac{1}{\pi_-}\Big)\Big| = \frac{|\tau|}{2}\Big| \frac{1}{\pi_+} - \frac{1}{\pi_-}\Big| < \frac{2q^{1/2}}{2(q + 1)} = \frac{1}{q^{1/2} + q^{-1/2}}.\qedhere
\]

\end{proof}

\begin{example}
Let $E/\FF_5$ be the elliptic curve with Weierstrass equation $y^2 = x^3 + x + 1$ over $\FF_5$. We quickly count the number of points in $E(\FF_5)$ by an exhaustive search and find
\[
    9 = \#E(\FF_5) = q + 1 - \tau = 5 + 1 + 3,
\]
hence the trace of Frobenius is $\tau = -3$. Thus $q + 1 + \tau = 5 + 1 - 3 = 3$. Therefore Theorem \ref{thm lattes periodic count} implies that the density of periodic points for $L_d$ with $d$ coprime to $5$ depends only on whether or not $d$ is divisible by 3. In particular,
\[
    \delta(L_d,5) = \begin{cases} 1 & \text{if }3 \nmid d\\ \frac{1}{6} & \text{if }3 \mid d. \end{cases}
\]
\end{example}

\begin{example}
Let $E/\FF_7$ be the elliptic curve with Weierstrass equation $y^2 = x^3 + x - 1$ over $\FF_7$. We compute
\[
    11 = \#E(\FF_7) = q + 1 - \tau = 7 + 1 + 3.
\]
Hence $\tau = -3$ and $q + 1 + \tau = 5$. Thus Theorem \ref{thm lattes periodic count} implies that the there are 4 cases for the density of periodic points of $L_d$ on $\PP^1(\FF_7)$ with $7 \nmid d$ depending on the greatest common divisor $(d,55)$ of $d$ and 55.
\[
    \delta(L_d,7) = \begin{cases} 
    1 & \text{if }(d,55) = 1,\\
    \frac{3}{4} & \text{if }(d,55) = 5,\\
    \frac{3}{8} & \text{if }(d,55) = 11,\\
    \frac{1}{8} & \text{if }(d, 55) = 55.
    \end{cases}
\]
\end{example}

\begin{cor}
\label{cor permutation characterization}
Let $E/\FF_q$ be an elliptic curve with trace of Frobenius $\tau$. Then the $d$th Latt\`es map associated to $E$ induces a permutation of $\PP^1(\FF_q)$ if and only if $(q + 1)^2 - \tau^2$ is coprime to $d$.
\end{cor}

\begin{proof}
Note that $L_d$ induces a permutation of $\PP^1(\FF_q)$ if and only if every point in $\PP^1(\FF_q)$ is periodic. Theorem \ref{thm lattes periodic count} implies that
\[
    \delta(L_d,q) = \frac{1}{2}\Big(\frac{1}{\pi_+} + \frac{1}{\pi_-}\Big) + \frac{\tau}{2(q + 1)}\Big(\frac{1}{\pi_+} - \frac{1}{\pi_-}\Big) \leq 1,
\]
with equality achieved if and only if
\[
    1 = \pi_\pm = \prod_{\ell \mid d} \ell^{v_\ell(q + 1 \pm \tau)},
\]
which is equivalent to $(q + 1)^2 - \tau^2 = (q + 1 - \tau)(q + 1 + \tau)$ being coprime to $d$.
\end{proof}

\begin{remark}
\label{remark kucuksakalli}
K\"u\c{c}\"uksakalli \cite{kucuksakalli} studied the value sets of Latt\`es maps associated to elliptic curves over $\FF_q$ which arise by starting with an elliptic curve with complex multiplication (CM) by the ring of integers in an imaginary quadratic field and reducing modulo a prime ideal $\pi$ of good reduction with norm $N(\pi) = q$. Their Corollary 2.8 gives a necessary and sufficient condition for such Latt\`es maps to induce permutations on $\PP^1(\FF_q)$ which is essentially equivalent to our Corollary \ref{cor permutation characterization}: Any elliptic curve $E/\FF_q$ which is not supersingular has endomorphism ring isomorphic to an order in an imaginary quadratic field \cite[Thm. V.3.1]{jhsec} and thus may be realized as the reduction of a CM elliptic curve defined over a number field at a prime of good reduction. While their result is stated for elliptic curves with CM by the full ring of integers in an imaginary quadratic field, the arguments appear to hold in greater generality.
\end{remark}

\section{Latt\`es periodic density in towers}

In this section we use Theorem \ref{thm lattes periodic count} to prove the following result on the density of periodic points for a fixed Latt\`es map in $\PP^1(\FF_{q^n})$ as $n$ varies. 

\begin{thm}
\label{thm periodic proportions}
Let $E/\FF_q$ be an elliptic curve, let $\tau_n$ be the trace of Frobenius of $E$ as an elliptic curve over $\FF_{q^n}$, and let $d$ be an integer coprime to $q$. For each positive integer $n$, there exists some $N$ depending on $q, \tau_1, d, n$ such that
\[
    \lim_{\substack{m\rightarrow \infty\\ m \equiv n \bmod cd^N}} \delta(L_d,q^m) = \frac{1}{2}\Big(\frac{1}{\pi_{n,+}} + \frac{1}{\pi_{n,-}}\Big),
\]
where
\[
    \pi_{n,\pm} := \prod_{\ell \mid d} \ell^{v_\ell(q^n + 1 \pm \tau_n)},
\]
the product is taken over all primes $\ell$ dividing $d$, $c := \lcm(\ell^2 - 1 : \ell \mid d \text{ is prime})$, and the limit is taken over all positive integers $m$ such that $m \equiv n \bmod cd^N$.
\end{thm}

\begin{proof}
Recall that the \emph{Hasse-Weil zeta function} of $E$ over $\FF_q$ is the formal power series
\[
    Z(E/\FF_q,x) := \exp\Big(\sum_{m\geq 1} \frac{\#E(\FF_{q^m})}{m}x^m\Big).
\]
The zeta function of $E$ may be expressed as a rational function \cite[Thm. V.2.4]{jhsec} given explicitly by
\[
    Z(E/\FF_q,x) = \frac{1 - \tau x + qx^2}{(1 - x)(1 - qx)}.
\]
Let $\alpha, \beta$ be the algebraic integers such that
\[
    1 - \tau x + qx^2 = (1 - \alpha x)(1 - \beta x).
\]
Then by taking a logarithm of $Z(E/\FF_q,x)$ and comparing coefficients we see that
\[
    \#E(\FF_{q^m}) = q^m + 1 - \alpha^m - \beta^m.
\]
Thus the trace of Frobenius of $E$ viewed as an elliptic curve over $\FF_{q^m}$ is $\tau_m = \alpha^m + \beta^m$.

Suppose $\ell$ is a prime dividing $d$. Let $|\cdot |_\ell$ be the $\ell$-adic absolute value on $\QQ$, let $\QQ_\ell$ be the $\ell$-adic completion of $\QQ$, and let $K/\QQ_\ell$ be a finite extension. Every unit $\gamma \in K$ may be expressed as $\gamma = \zeta (1 + \delta \lambda)$ for some root of unity $\zeta$, some $\delta$ with $|\delta|_\ell \leq 1$, and some uniformizer $\lambda$. Furthermore the order of $\zeta$ divides $\ell^f - 1$ where $f$ is the residue degree of $K/\QQ_\ell$ (see Koblitz\cite[\textsection III.3]{koblitz}.) The binomial theorem implies that $m \mapsto (1 + \delta\lambda)^m$ is an $\ell$-adically continuous function of $m$ \cite[\textsection II.2]{koblitz}. Hence $m \mapsto \gamma^m = \zeta^m (1 + \delta\lambda)^m$ is $\ell$-adically continuous when restricted to $m$ in a residue class modulo $\ell^f - 1$.

Now let $K := \QQ_\ell(\alpha, \beta)$ be the splitting field of $1 - \tau x + q x^2$ over $\QQ_\ell$. Since $\alpha, \beta$ are algebraic integers such that $\alpha \beta = q$ and $q$ is coprime to $\ell$ by assumption, it follows that $|q|_\ell = |\alpha|_\ell = |\beta|_\ell = 1$. The residue degree of $K/\QQ_\ell$ is at most 2. Therefore, with $n$ fixed, the functions
\[
    \ve_+(m) := q^m + 1 + \alpha^m + \beta^m, \hspace{.25in} \ve_-(m) := q^m + 1 - \alpha^m - \beta^m
\]
are $\ell$-adically continuous on the set of all integers $m$ such that $m \equiv n \bmod \ell^2 - 1$. Thus there exists some $N_\ell$ depending on $q, \tau, \ell$ such that $m \equiv n \bmod (\ell^2 - 1)\ell^{N_\ell}$ implies
\[
    |\ve_\pm(m) - \ve_\pm(n)|_\ell < |\ve_\pm(n)|_\ell.
\]
The ultrametric property of $|\cdot|_\ell$ implies that $|\ve_\pm(m)|_\ell = |\ve_\pm(n)|_\ell$ for all such $m$. This is equivalent to
\[
    v_\ell(q^m + 1 \pm \tau_m) = v_\ell(q^n + 1 \pm \tau_n).
\]

Let $N := \max_{\ell \mid d} N_\ell$ and let $c : \lcm(\ell^2 - 1 : \ell \mid d \text{ is prime})$. Hence if $m \equiv n \bmod cd^N$, then $m \equiv n \bmod (\ell^2 - 1)\ell^{N_\ell}$ for each prime $\ell \mid d$. Therefore, Theorem \ref{thm lattes periodic count} implies that for any such $m$,
\[
    \Big|\delta(L_d,q^m) - \frac{1}{2}\Big(\frac{1}{\pi_{n,+}} + \frac{1}{\pi_{n,-}}\Big)\Big| < \frac{1}{q^{m/2} + q^{-m/2}}.
\]
We conclude that
\[
    \lim_{\substack{m\rightarrow \infty\\ m \equiv n \bmod cd^N}} \delta(L_d,q^m) = \frac{1}{2}\Big(\frac{1}{\pi_{n,+}} + \frac{1}{\pi_{n,-}}\Big).\qedhere
\]

\end{proof}

\section{Latt\`es periodic densities when $\tau = 0$}
In this section we consider a special family of elliptic curves where the periodic density in $\PP^1(\FF_{q^n})$ of $L_\ell$, with $\ell$ prime, may be computed in terms of the $\ell$-adic valuation of $n$ and $q^e - 1$, where $e$ is the multiplicative order of $q$ modulo $\ell$.

\begin{thm}
\label{thm special supersingular evaluation}
Let $E/\FF_q$ be an elliptic curve with trace of Frobenius $\tau = 0$. If $\ell > 2$ is a prime not dividing $q$ and $e$ is the multiplicative order of $q$ modulo $\ell$, then set $w_n := v_\ell(q^e - 1) + v_\ell(n)$. Then
\[
    \delta(L_\ell,q^n)
    =
    \begin{cases}
         \ell^{-w_n} & \text{if $n$ is odd, $e \mid 2n$, and $e \nmid n$,}\\
         \frac{1}{2}(1 + \ell^{-2w_n}) + \frac{\epsilon}{q^{n/2} + q^{-n/2}}(1 - \ell^{-2w_n}) & \text{if $n$ is even and $e \mid n$,}\\
         1 & \text{otherwise,}
    \end{cases}
\]
where $\epsilon$ is a sign determined by
\[
\epsilon = 
    \begin{cases}
        +1 & \text{if $n\equiv 2 \bmod 4$ and $e\nmid n/2$, or $n\equiv 0 \bmod 4$ and $e \mid n/2$,} \\
        -1 & \text{if $n\equiv 2 \bmod 4$ and $e\mid n/2$, or $n\equiv 0 \bmod 4$ and $e \nmid n/2$.}
    \end{cases}
\]
Hence if $e$ does not divide $2n$, then $L_\ell$ induces a permutation on $\PP^1(\FF_{q^n})$.
\end{thm}

\begin{remark}
Recall that if $\FF_q$ has characteristic $p$, then an elliptic curve $E$ over $\FF_q$ is said to be \emph{supersingular} if $\tau$ is divisible by $p$. Supersingular elliptic curves have many exceptional properties, see Silverman \cite[\textsection V.3]{jhsec}. In particular, if $\tau = 0$, then $E$ must be supersingular. On the other hand, if $E$ is supersingular and $q = p \geq 5$, then the Hasse bound $|\tau| \leq 2\sqrt{p}$ implies that $\tau = 0$. Thus Theorem \ref{thm special supersingular evaluation} applies to all Latt\`es maps $L_\ell$ associated to supersingular elliptic curves over $\FF_p$ when $p\geq 5$.
\end{remark}

\begin{lem}
\label{lem valuation of q^n - 1}
Suppose that $q \in \overline{\QQ}_\ell$ is an element such that $v_\ell(q - 1) > \frac{1}{\ell - 1}$. Then
\[
    v_\ell(q^n - 1) = v_\ell(q - 1) + v_\ell(n).
\]
In particular, this holds when $q$ is an integer such that $q \equiv 1 \bmod \ell$.
\end{lem}

\begin{proof}
Let $\zeta_n$ denote a primitive $n$th root of unity. Recall that
\[
    |N(1 - \zeta_n)| = 
    \begin{cases}
        p & \text{if $n = p^k$ for some prime $p$ and $k\geq 1$,}\\
        1 & \text{if $n$ is not a prime power,}
    \end{cases}
\]
where $N : \QQ(\zeta_n) \rightarrow \QQ$ is the norm function (see, for example, Lang \cite[Chp. IV, \textsection 1]{lang}.) Recall that the degree of the field extension $\QQ(\zeta_{\ell^m})/\QQ$ is $\varphi(\ell^m) := \ell^m - \ell^{m-1}$. Thus if $n$ is coprime to $\ell$, then $v_\ell(1 - \zeta_n^k) = 0$ for all $k\not\equiv 0 \bmod n$, and 
\[
    v_\ell(1 - \zeta_{\ell^m}^k) = \frac{1}{[\QQ(\zeta_{\ell^m}) : \QQ]}v_\ell(N(1 - \zeta_{\ell^m}^k)) = \frac{1}{\varphi(\ell^m)}v_\ell(N(1 - \zeta_{\ell^m})) = \frac{1}{\varphi(\ell^m)},
\]
for any $k \not\equiv 0 \bmod \ell$. Thus, $v_\ell(1 - \zeta_n^k) \leq \frac{1}{\ell - 1}$ for all $k\not\equiv 0 \bmod n$ and for all $m\geq 1$,
\[
    \sum_{k \in (\ZZ/(\ell^m))^\times}v_\ell(1 - \zeta_{\ell^m}^k) = 1.
\]

If $n\geq 1$, the factorization $q^n - 1 = \prod_{k=0}^{n-1}(q - \zeta_n^k)$ implies that
\[
    v_\ell(q^n - 1) - v_\ell(q - 1) = \sum_{k=1}^{n-1}v_\ell(q - \zeta_n^k)
    = \sum_{k=1}^{n-1}\min(v_\ell(q - 1), v_\ell(1 - \zeta_n^k))
    = \sum_{k=1}^{n-1}v_\ell(1 - \zeta_n^k),
\]
where the second and third equalities follow from our assumption that $v_\ell(q - 1) > \frac{1}{\ell - 1} \geq v_\ell(1 - \zeta_n^k)$. Observe that each $\zeta_n^k$ has a unique expression as $\zeta_d^j$ where $d \mid n$ and $j$ is a unit modulo $d$. Thus, if $v_\ell(n) = m$, then
\[
    v_\ell(q^n - 1) - v_\ell(q - 1) = \sum_{d\mid n}\sum_{j\in (\ZZ/(d))^\times}v_\ell(1 - \zeta_d^j)
    = \sum_{i=1}^m \sum_{j\in (\ZZ/(\ell^i))^\times}v_\ell(1 - \zeta_{\ell^i}^j)
    = m.\qedhere
\]
\end{proof}

\begin{proof}[Proof of Theorem \ref{thm special supersingular evaluation}]
Let $E$ be an elliptic curve over $\FF_q$ with $\tau = 0$. Let $\tau_n$ be the trace of Frobenius of $E$ over $\FF_{q^n}$. Then as discussed in the proof of Theorem \ref{thm periodic proportions}, we have $\tau_n = \alpha^n + \beta^n$ where
\[
    (1 - \alpha x)(1 - \beta x) = 1 - \tau x + qx^2 = 1 - qx^2.
\]
Hence $\alpha, \beta = \pm i \sqrt{q}$ and
\[
    \tau_n = (i^n + (-i)^n)q^{n/2} = 
    \begin{cases}
        \hspace{1.4em}0 & \text{if $n$ is odd,}\\
        2(-1)^{n/2}q^{n/2} & \text{if $n$ is even.}
    \end{cases}
\]
Therefore,
\[
    q^n + 1 \pm \tau_n = 
    \begin{cases}
        q^n + 1 & \text{if $n$ is odd,}\\
        (q^{n/2} \pm (-1)^{n/2})^2 & \text{if $n$ is even.}
    \end{cases}
\]

If $n$ is odd, then Theorem \ref{thm lattes periodic count} implies that
\begin{align*}
    \delta(L_\ell,\FF_{q^n}) 
    &= \frac{1}{2}\Big(\ell^{-v_\ell(q^n + 1 + \tau_n)} + \ell^{-v_\ell(q^n + 1 - \tau_n)}\Big) + \frac{\tau_n}{q^n + 1}\Big(\ell^{-v_\ell(q^n + 1 + \tau_n)} - \ell^{-v_\ell(q^n + 1 - \tau_n)}\Big)\\
    &= \ell^{-v_\ell(q^n + 1)}.
\end{align*}
Since $\ell > 2$ and $(q^n + 1) + (q^n - 1) = 2q^n$, at most one of $v_\ell(q^n + 1)$ and $v_\ell(q^n - 1)$ is positive. Note that
\[
    v_\ell(q^n + 1) + v_\ell(q^n - 1) = v_\ell(q^{2n} - 1),
\]
and $v_\ell(q^{2n} - 1) > 0$ if and only if $e \mid 2n$. Similarly, $v_\ell(q^n - 1) > 0$ if and only if $e \mid n$. Hence $v_\ell(q^n + 1) > 0$ is equivalent to $e \mid 2n$ and $e \nmid n$. In that case Lemma \ref{lem valuation of q^n - 1} implies that
\begin{align*}
    v_\ell(q^n + 1) &= v_\ell(q^{2n} - 1) - v_\ell(q^n - 1)\\
    &= v_\ell(q^{2n} - 1)\\
    &= v_\ell(q^e - 1) + v_\ell(2n/e)\\
    &= v_\ell(q^e - 1) + v_\ell(n)\\
    &= w_n.
\end{align*}
The fourth equality follows from the fact that $e$ is a divisor of $\varphi(\ell) = \ell - 1$, hence $v_\ell(e) = 0$. Therefore, if $n$ is odd,
\[
    \delta(L_\ell,q^n) =
    \begin{cases}
         \ell^{-w_n} & \text{if $e\mid 2n$ and $e\nmid n$,}\\
         1 & \text{otherwise.}
    \end{cases}
\]

Next suppose that $n$ is even. Then by Theorem \ref{thm lattes periodic count},
\[
    \delta(L_\ell,q^n)
    = \frac{1}{2}\Big(\ell^{-2a_{n}^+} + \ell^{-2a_{n}^-}\Big) + \frac{(-1)^{n/2}}{q^{n/2} + q^{-n/2}}\Big(\ell^{-2a_n^+} - \ell^{-2a_n^-}\Big),
\]
where $a_n^\pm = v_\ell(q^{n/2} \pm (-1)^{n/2})$. Since
\[
    (q^{n/2} - (-1)^{n/2}) + (q^{n/2} + (-1)^{n/2}) = 2q^{n/2},
\]
and $\ell > 2$ by assumption, it follows that at most one of $a_n^\pm$ is positive. Then
\[
    a_n^+ + a_n^- = v_\ell(q^n - 1)
\]
is positive if and only if $e\mid n$. Suppose that $e\mid n$. Lemma \ref{lem valuation of q^n - 1} implies that
\[
    v_\ell(q^n - 1) = v_\ell(q^e - 1) + v_\ell(n/e) = w_n.
\]
Furthermore, we have $v_\ell(q^{n/2} - 1) > 0$ if and only if $e \mid n/2$, hence $v_\ell(q^{n/2} + 1) > 0$ if and only if $e \nmid n/2$. Putting it all together, for $n$ even and $e \mid n$ we have
\[
    \delta(L_\ell,q^n)
    = \frac{1}{2}\Big(1 + \ell^{-2w_n}\Big) + \frac{\epsilon}{q^{n/2} + q^{-n/2}}\Big(1 - \ell^{-2w_n}\Big),
\]
where
\[
    \epsilon = 
    \begin{cases}
        +1 & \text{if $n\equiv 2 \bmod 4$ and $e\nmid n/2$, or $n\equiv 0 \bmod 4$ and $e \mid n/2$,} \\
        -1 & \text{if $n\equiv 2 \bmod 4$ and $e\mid n/2$, or $n\equiv 0 \bmod 4$ and $e \nmid n/2$.}
    \end{cases}
\]
In all other cases, $\delta(L_\ell,q^n) = 1$. For example, this holds when $e$ does not divide $2n$.
\end{proof}

\bibliographystyle{plain}
\bibliography{Ref}

\begin{thebibliography}{10}

\bibitem{juul}
Jamie Juul.
\newblock Fixed point proportions for {G}alois groups of non-geometric iterated
  extensions.
\newblock {\em Acta Arith.}, 183(4):301--315, 2018.

\bibitem{JKMT}
Jamie Juul, P\"ar Kurberg, Kalyani Madhu, and Tom~J. Tucker.
\newblock Wreath products and proportions of periodic points.
\newblock {\em Int. Math. Res. Not.}, 13:3944--3969, 2016.

\bibitem{koblitz}
Neal Koblitz.
\newblock {\em $p$-adic Numbers, $p$-adic Analysis, and Zeta-Functions},
  volume~58 of {\em Graduate Texts in Mathematics}.
\newblock Springer-Verlag, Berlin, 2nd edition, 1984.

\bibitem{kucuksakalli}
\"Omer K\"u\c{c}\"uksakalli.
\newblock Value sets of {L}att\`es maps over finite fields.
\newblock {\em J. Number Theory}, 143:262--278, 2014.

\bibitem{lang}
Serge Lang.
\newblock {\em Algebraic Number Theory}, volume 110 of {\em Graduate Texts in
  Mathematics}.
\newblock Springer-Verlag, 1986.

\bibitem{madhu}
Kalyani Madhu.
\newblock {\em Galois Theory and Polynomial Orbits}.
\newblock PhD thesis, University of Rochester, 2011.

\bibitem{manes2013periodic}
Michelle Manes and Bianca Thompson.
\newblock Periodic points in towers of finite fields for polynomials associated
  to algebraic groups.
\newblock {\em Rocky Mountain J. Math.}, 49(1):171--197, 2019.

\bibitem{shahu}
Min Sha and Su~Hu.
\newblock Monomial dynamical systems of dimension one over finite fields.
\newblock {\em Acta Arith.}, 148(4):309--331, 2011.

\bibitem{jhsec}
Joseph~H. Silverman.
\newblock {\em The Arithmetic of Elliptic Curves}, volume 106 of {\em Graduate
  Texts in Mathematics}.
\newblock Springer-Verlag, 1992.

\bibitem{ads}
Joseph~H. Silverman.
\newblock {\em The Arithmetic of Dynamical Systems}, volume 241 of {\em
  Graduate Texts in Mathematics}.
\newblock Springer-Verlag, 2007.

\bibitem{ugolini}
Simone Ugolini.
\newblock Functional graphs of rational maps induced by endomorphisms of
  ordinary elliptic curves over finite fields.
\newblock {\em Period. Math. Hungar.}, 77(2):237--260, 2018.

\end{thebibliography}

\end{document}